\documentclass[11pt,a4paper]{amsart}
\usepackage{amssymb,amsmath}
\usepackage{amsthm, dsfont, a4}
\usepackage{hyperref}
\usepackage[all]{xy}

\usepackage[active]{srcltx}

\newcommand\comment[1]{}

\newcommand\classification[2][]{%
  \gdef\@classification{%
    \href{http://www.ams.org/msc/}%
{\textit{2000 Mathematics Subject Classification}} \ignorespaces#2\unskip}}

\leftmargin  \leftmargini
\comment{

} 

\def\NN{{\mathbb N}}

\def\G{{\mathcal G}}
\def\H{{\mathcal H}}

\def\O{{\mathcal O}}

\def\fH{{\mathfrak H}}
\def\fM{{\mathfrak M}}

\def\id{{\rm id}}

\def\cat#1{{\sf #1}}

\def\vect#1{\text{\boldmath $#1$\unboldmath}} 
\def\ie{i.\,e.}

\def\isom{\cong}
\def\congr{\equiv}

\def\LRa{\Leftrightarrow}

\DeclareMathOperator{\Ker}{Ker}

\DeclareMathOperator{\Aut}{Aut}

\DeclareMathOperator{\GL}{GL}

\DeclareMathOperator{\Gal}{Gal}
\DeclareMathOperator{\Quot}{Quot}

\DeclareMathOperator{\Spec}{Spec}
\def\uGal{\underline{\Gal}}

\def\markdef{\bf }

\theoremstyle{plain}
\newtheorem{thm}{Theorem}[section]

\newtheorem{lem}[thm]{Lemma}

\theoremstyle{definition}
\newtheorem{defn}[thm]{Definition}

\newtheorem{rem}[thm]{Remark}
\newtheorem{exam}{Example}
\newenvironment{notation}{{\bf Notation}\it}

\begin{document}

\title[Torsion group schemes]{Torsion group
  schemes as iterative differential Galois groups}
\author{Andreas Maurischat}
\address{\rm {\bf Andreas Maurischat}, Lehrstuhl A f\"ur Mathematik, RWTH Aachen University, Germany }
\email{\sf andreas.maurischat@matha.rwth-aachen.de}


\classification{12H20, 12F12, 13B05}

\date{\today}

\begin{abstract}
We show that torsion group schemes of abelian varieties in positive characteristic occur as iterative differential Galois groups of extensions of iterative differential fields. The main part is to find computable criteria when higher derivations are iterative derivations, and furthermore when an iterative derivation
on the function field of an abelian variety is compatible with the addition map.
For an explicit example, we give a construction of (a family of) such iterative derivations on the function field of an elliptic curve in characteristic two.
\end{abstract}

\maketitle

\section{Introduction}

For transcendental field extensions $L/F$ the group of automorphisms of $L$ over $F$ is huge and one is
far from obtaining a Galois correspondence. By considering derivations on the fields (resp.
iterative derivations in positive characteristic), one obtains a natural subgroup of all
automorphisms, namely those automorphisms which commute with the (iterative) derivation. These
automorphisms are called (iterative) differential automorphisms. In special cases, the group of
(iterative) differential automorphisms form a linear algebraic group and one has a Galois
correspondence between Zariski-closed subgroups and intermediate differential fields.
In Picard-Vessiot theory one considers such cases. Here the extension
field $L$ is obtained as the solution field of a
linear (iterative) differential equation over the differential field $F$, quite analogous to the
classical Galois theory where the extension fields are obtained as solution fields of algebraic
equations.
By considering the automorphism group not as a group, but as a group scheme, one can deal with
nonnormal and even inseparable iterative differential extensions (see
\cite{td:tipdgtfrz} and \cite{am:gticngg}, Sect.10).
Moreover, this also applies to finite ID-extensions, and one can even obtain an infinitesimal group
scheme as ID-Galois group scheme (cf.~\cite{am:igsidgg}).

In this article, we will consider special finite group schemes, namely the torsion group schemes of
an abelian variety. More precisely, we give iterative differential field extensions having as
ID-Galois group scheme the torsion group scheme of an abelian variety.
Throughout the article we will stick to positive characteristic.

The rough idea for getting the $n$-torsion scheme
$A[n]$ of an abelian variety $A$ over a perfect field $C$ as ID-Galois group scheme is the
following.
Starting with the abelian variety $A$ over $C$ we consider the function field $L$ of $A_{C(t)}$ (i.e. of
$A$ after base change to $C(t)$) as an extension of the rational function field $C(t)$. The field
$C(t)$ comes with a standard iterative derivation with respect to $t$ (the characteristic $p$-analog
of the derivation $\frac{\partial}{\partial t}$), and this iterative derivation is then extended to
an iterative derivation on $L$. By taking care that this extension fulfills the appropriate
conditions, one guarantees that the torsion group scheme $A[n]$ indeed acts on $L$ by
ID-automorphisms. Hence by Picard-Vessiot theory, one obtains $A[n]$ as the iterative differential
Galois group of $L$ over $L^{A[n]}$, the fixed field under $A[n]$. To be more precise, one should
say that the group scheme acts by functorial automorphisms, i.e. $D$-rational points act as
ID-automorphisms on the total quotient ring $\Quot(L\otimes_C D)$.

The article is structured as follows.
In Section \ref{sec:basics}, we give the basic notation and some basic properties which will be used
in the calculations later on. Furthermore, we give a short summary of the Picard-Vessiot theory
used in this article. The theoretical considerations for obtaining the torsion group scheme of an
abelian scheme as ID-Galois group are given in Sections \ref{sec:it-der-on-ab-schemes} and
\ref{sec:galois-groups}. The main theorems are Theorem \ref{thm:commuting with rho} giving a
necessary and sufficient condition for the iterative derivation on the function field of an abelian
variety to ``commute'' with the addition map, as well as Theorem \ref{thm:torsion as galois} stating
that the torsion group schemes are the ID-Galois group schemes over an
appropriate subfield when the iterative derivation satisfies the previous conditions.

In the last sections we do explicite calculations. While Section \ref{sec:ID-extensions} deals with
the extension of an iterative derivation to an overfield in general, Section \ref{sec:example} is
dedicated to the example of an elliptic curve in characteristic $2$. In this case, we
give recursive formulas for constructing an iterative derivation on the function field which
satisfies the previously stated conditions (cf.~Theorem~\ref{thm:strong formula}).

\medskip


\section{Basic notation}\label{sec:basics}

All rings are assumed to be commutative with unit.

We will use the following notation (see also \cite{am:igsidgg}). 
A {\markdef higher derivation} (HD for short) on a ring $R$ is a homomorphism of rings $\theta:R\to
R[[T]]$, such that $\theta(r)|_{T=0}=r$ for all $r\in R$. If there is need to emphasis the extra
variable $T$ or if we use another name for the variable, we add a subscript to $\theta$, i.e.~denote
the higher derivation by $\theta_T$ (resp. $\theta_U$ if the variable is named $U$).

A higher derivation is called
an {\markdef iterative derivation} (ID for short) if for all $i,j\geq
0$, $\theta^{(i)}\circ \theta^{(j)}=\binom{i+j}{i}\theta^{(i+j)}$,
where the maps $\theta^{(i)}:R\to R$ are defined by
$\theta(r)=:\sum_{i=0}^\infty \theta^{(i)}(r)T^i$.
The pair $(R,\theta)$ is then called an {\markdef HD-ring} (resp. {\markdef ID-ring}) and
$C_R:=\{ r \in R\mid \theta(r)=r\}$ is called the {\markdef ring of
  constants } of $(R,\theta)$. 
An HD/ID-ring which is a field is called an {\markdef HD/ID-field}.
Higher derivations and iterative derivations are extended to localisations by
 $\theta(\frac{r}{s}):=\theta(r)\theta(s)^{-1}$ and to tensor products
 by 
$$\theta^{(k)}(r\otimes s)=\sum_{i+j=k} \theta^{(i)}(r)\otimes
\theta^{(j)}(s)$$ 
for all $k\geq 0$.

Given a homomorphism of rings $f:R\to S$, we often consider the $T$-linear extension of $f$ to a
homomorphism $R[[T]]\to S[[T]]$ of the power series rings. This map will be denoted by $f[[T]]$.
Given two HD-rings $(R,\theta)$ and $(S,\tilde{\theta})$. A homomorphism of rings $f:R\to S$ is
called an {\markdef HD-homomorphism} (resp. ID-homomorphism if $R$ and $S$ are ID-rings) if
$\tilde{\theta}\circ f=f[[T]]\circ \theta$.
As a special case of a homomorphism $f[[T]]$, we have the homomorphism $\theta_U[[T]]:R[[T]]\to
R[[T,U]]$ induced by the higher derivation $\theta_U:R\to R[[U]]$ on $R$.
A short calculation shows (cf. \cite{ar:icac}) that a higher derivation $\theta$ on $R$ is an
iterative derivation if and only if the following diagram commutes
\centerline{\xymatrix@+10pt{
R \ar[r]^{\theta_U} \ar[d]_{\theta_T} & R[[U]]
\ar[d]^{U\mapsto U+T} \\
R[[T]] \ar[r]^{\theta_U[[T]]} & R[[U,T]],
}}
or in other terms $\theta_U[[T]]\circ \theta_T=\theta_{T+U}$.

\begin{exam}\label{ex:ID-fields} (cf. \cite{am:gticngg})
\begin{enumerate}
\item\label{item:der by t} For any field $C$ and $F:=C(t)$, the homomorphism of $C$-algebras
$\theta:F \to F[[T]]$
given by $\theta(t):=t+T$ is an iterative derivation on $F$ with field of constants $C$. This
iterative derivation will be called the {\markdef iterative derivation with respect to~$t$}.
\item For any ring $R$, there is the {\markdef trivial} iterative derivation on $R$ given by
$\theta_0:R\to R[[T]],r\mapsto r\cdot T^0$. Obviously, the ring of constants of $(R,\theta_0)$ is
$R$ itself.
\item If $(F,\theta)$ is an HD-field and $L\geq F$ is a finite separable field extension, then
$\theta$ can be uniquely extended to a higher derivation on $L$. If the higher derivation $\theta$
is an iterative derivation, then the extension to $L$ is also an iterative derivation.
\item Let $(F,\theta)$ be an HD-field, $L/F$ a finitely generated separable field extension and
$x_1,\dots, x_k$ a separating transcendence basis of $L$ over $F$ (i.e. $F(x_1,\dots, x_k)/F$ is
purely transcendental and $L/F(x_1,\dots, x_k)$ is finite separable). Using the previous
example, it is easy to see that
any choice of elements $\xi_{i,n}\in L$ ($i=1,\dots,k$ and $n\geq 1$) defines a unique higher
derivation $\theta_L$ on $L$ extending $\theta$ and satisfying  $\theta_L(x_i)= x_i+
\sum_{n=1}^\infty \xi_{i,n} T^n$ for all $i=1,\dots, k$.
\end{enumerate}
\end{exam}
We now summarize some well known formulas for higher derivations and iterative derivations in characteristic $p>0$ which will
be used later on:

\begin{lem}\label{lem:known-stuff} \
\begin{enumerate}
\item $\theta^{(j)}(x^p)=0$ if $p$ does not divide $j$ and
$\theta^{(j)}(x^p)=\left(\theta^{(j/p)}(x)\right)^p$ if $p$ divides~$j$.
\item If $m=m_0+m_1p+\dots + m_kp^k$ and $n=n_0+n_1p+\dots + n_kp^k$ where $m_i,n_i\in \{0,\dots,
p-1\}$ then $$\binom{m}{n}\congr \binom{m_0}{n_0}\cdot \binom{m_1}{n_1}\cdots \binom{m_k}{n_k}
\mod{p}.$$
\item If $\theta$ is iterative, then $(\theta^{(j)})^p=0$ for all $j$.
\item \label{item:expansion} Let $m=m_0+m_1p+\dots + m_kp^k$ where $m_i\in \{0,\dots,
p-1\}$.  If $\theta$ is iterative, then all the $\theta^{(p^i)}$ commute with each other, and
$$\theta^{(m)}=\frac{1}{m_0!\cdot m_1!\cdots m_k!} (\theta^{(1)})^{m_0}\circ  (\theta^{(p)})^{m_1}\circ \dots \circ  (\theta^{(p^k)})^{m_k}.$$
\end{enumerate}
\end{lem}

\begin{notation}
Let $(L,\theta)$ be an HD-field of characteristic $p>0$, and $k\in \NN\cup \{\infty\}$.
We say that ``for $x\in L$ the iteration rule holds up to level $k$'' if for all $i,j\in\NN$ satisfying $i+j\leq k$ one has
$$\theta^{(i)}\circ \theta^{(j)}(x)
=\binom{i+j}{i}\theta^{(i+j)}(x),$$
or equivalently if
$$\theta_U[[T]] \bigl( \theta_T(x)\bigr)\equiv \theta_{T+U}(x) \mod (U^{k+1-j}T^j\mid 0\leq j\leq k+1).$$
We say that ``the iteration rule holds on $L$ up to level $k$'' if the iteration rule holds up to $k$ for all $x\in L$.
\end{notation}

\begin{lem}\label{lem:Things to show}
Let $(L,\theta)$ be an HD-field of characteristic $p>0$.
\begin{enumerate}
\item \label{item:subfield} For $k\in \NN\cup \{\infty\}$, the set of elements $x\in L$ for
which the iteration rule holds up to level $k$ is a subfield of $L$.
\item \label{item:level-extended} Assume that for fixed $\ell\geq 0$ the iteration rule holds on $L$ up to level $p^\ell$, then for all $0\leq k,m<p^\ell$ such that $k+m\geq p^\ell$, one has $$\theta^{(k)}\circ \theta^{(m)}=0=\binom{k+m}{k}\theta^{(k+m)}.$$
\item \label{item:cancellation} Assume that for fixed $\ell\geq 0$ the iteration rule holds on $L$ up to level $p^\ell$, and that $L$
contains an element $t$ satisfying $\theta(t)=t+T$. Then
for all $x\in L$ and all $0<r<p^\ell$ one has:
$$\theta^{(r)}\left(\sum_{m=0}^{p^\ell-1} \theta^{(m)}(x)(-t)^m\right)=0$$
\end{enumerate}
\end{lem}

\begin{proof}
(\ref{item:subfield}) The set under consideration is just the equalizer of the ring homomorphisms
$\theta_U[[T]]\circ \theta_T:L\to L[[T,U]/(U^{k+1-j}T^j\mid 0\leq j\leq k+1)$ and $\theta_{T+U}$. Hence, it is a subfield of~$L$.

(\ref{item:level-extended}) As $k,m<p^\ell$ and $k+m\geq p^\ell$, the binomial coefficient $\binom{k+m}{k}$ equals $0$ in characteristic $p$. Hence, the right hand side of the equation equals $0$. For proving that the left hand side equals zero, it is sufficient to consider the case where $k=p^j$ for some $j<\ell$, as any $\theta^{(k)}$ is  a composition of those up to a non-zero constant.
Let $m':=m+p^j-p^\ell$. By assumption on $m$ and $p^j$, we have $0\leq m'<p^j$. As $m-m'=p^\ell-p^j$ is divisible by $p^j$, $m'$ is the first part of the $p$-adic expansion of $m$ up to $p^{j-1}$ and $m-m'$ is the second part. Hence, by the previous lemma $\binom{m}{m'}=1$ in characteristic $p$. As the iteration rule holds on $L$ up to level $p^\ell$, and as $k=p^j<p^\ell$ one gets
$$\theta^{(p^j)}\circ \theta^{(m)}= \theta^{(p^j)}\circ  \theta^{(p^\ell-p^j)} \circ \theta^{(m')}=\binom{p^\ell}{p^j} \theta^{(p^\ell)}\circ \theta^{(m')}=0,$$
since $\binom{p^\ell}{p^j}=0$.

(\ref{item:cancellation}) This is a more complicated, but straightforward calculation:
\begin{eqnarray*}
&\theta^{(r)}&\left(\sum_{m=0}^{p^\ell-1} \theta^{(m)}(x)(-t)^m\right) 
\quad = \quad \sum_{m=0}^{p^\ell-1} \sum_{k=0}^r  \theta^{(r-k)}\left(\theta^{(m)}(x)\right)(-1)^m\theta^{(k)}(t^m) \\
&\stackrel{{\rm by} (\ref{item:level-extended})}{=}{}&  \sum_{k=0}^r \sum_{m=k}^{p^\ell-1-(r-k)} \binom{m+r-k}{m}\theta^{(m+r-k)}(x) \cdot (-1)^m\binom{m}{k}t^{m-k} \\
&=& \sum_{k=0}^r \sum_{m'=0}^{p^\ell-1-r} \binom{m'+r}{m'+k}\binom{m'+k}{k}  (-1)^{m'+k}\theta^{(m'+r)}(x) t^{m'} \\
&\stackrel{(\star)}{=}{}& \sum_{m'=0}^{p^\ell-1-r} \binom{m'+r}{r} \underbrace{\left( \sum_{k=0}^r \binom{r}{k} (-1)^k \right)}_{=0}  (-1)^{m'}\theta^{(m'+r)}(x) t^{m'} \\
&=& 0.
\end{eqnarray*}
Equation $(\star)$ holds, as both products $\binom{m'+r}{m'+k}\binom{m'+k}{k}$ and $ \binom{m'+r}{r}\binom{r}{k}$ count the number of possibilities of splitting a set of cardinality $m'+r$ into three disjoint subsets of cardinalities $k$, $m'$ and $r-k$ respectively.
\end{proof}

\subsection*{Picard-Vessiot theory}

We now recall some definitions from Picard-Vessiot theory. $(F,\theta)$ denotes some ID-field with
constants $C$.
\begin{defn}
Let $A=\sum_{k=0}^\infty A_k T^k\in \GL_n(F[[T]])$ be a matrix with
$A_0=\mathds{1}_n$ and for all $k,l\in\NN$, 
$\binom{k+l}{l}A_{k+l}=\sum_{i+j=l} \theta^{(i)}(A_{k})\cdot A_j.$
An equation
$$\theta(\vect{y})=A\vect{y},$$
where $\vect{y}$ is a vector of indeterminants, is called an {\markdef
  iterative differential equation (IDE)}.
\end{defn}

\begin{rem}\label{ide-condition}
The condition on the $A_{k}$ is equivalent to the condition that
$\theta^{(k)}(\theta^{(l)}(Y_{ij}))=
\binom{k+l}{k}\theta^{(k+l)}(Y_{ij})$ holds
for a matrix $Y=(Y_{ij})_{1\leq i,j\leq n}\in\GL_n(E)$ satisfying $\theta(Y)=AY$, where $E$ is some
ID-extension of
$F$. (Such a $Y$ is called a {\markdef fundamental solution matrix}) . The condition
$A_{0}=\mathds{1}_n$ is equivalent to
$\theta^{(0)}(Y_{ij})=Y_{ij}$, and already implies that the matrix $A$ is
invertible.
\end{rem}

\begin{defn}
An ID-ring $(R,\theta_R)\geq (F,\theta)$ is called a {\markdef
  Picard-Vessiot ring} (PV-ring) for the IDE $\theta(\vect{y})=A\vect{y}$, if
the following hold: 
\begin{enumerate}
\item $R$ is an ID-simple ring, i.e.~has no nontrivial $\theta_R$-stable ideals.
\item There is a fundamental solution matrix $Y\in\GL_n(R)$, \ie{}, an invertible
  matrix satisfying $\theta(Y)=AY$.
\item As an $F$-algebra, $R$ is generated by the coefficients of $Y$
  and by $\det(Y)^{-1}$.
\item $C_R=C_F=C$.
\end{enumerate}
The quotient field $E=\Quot(R)$ (which exists, since such a PV-ring
is always an integral domain) is called a {\markdef
  Picard-Vessiot field} (PV-field) for the IDE
$\theta(\vect{y})=A\vect{y}$.\footnote{The PV-rings and PV-fields defined
here were called pseudo Picard-Vessiot rings (resp. pseudo Picard-Vessiot fields)
in \cite{am:gticngg} and \cite{am:igsidgg}. This definition, however, is the
most natural generalisation 
of PV-rings and
PV-fields to non algebraically closed fields of constants.}
\end{defn}

For a PV-ring $R/F$ one defines the functor
$$\underline{\Aut}^{ID}(R/F): (\cat{Algebras} / C) \to (\cat{Groups}),
D\mapsto \Aut^{ID}(R\otimes_C D/F\otimes_C D)$$
where $D$ is equipped with the trivial iterative derivation.
In \cite{am:gticngg}, Sect.~10, it is shown that this functor is
represent\-able by a $C$-algebra of
finite type, and hence, is an affine group scheme of finite type over 
$C$. This group scheme is called the (iterative differential) {\markdef Galois
  group scheme} of the extension $R$ over $F$ -- denoted by
$\underline{\Gal}(R/F)$ --, or also, the
Galois group scheme of  the extension $E$ over $F$, $\underline{\Gal}(E/F)$,
where
$E=\Quot(R)$ is the corresponding PV-field.

Furthermore, $\Spec(R)$ is a $(\uGal(R/F) \times_{C}F)$-torsor and the
corresponding isomorphism of rings
\begin{equation}\label{eq:torsor-iso} \gamma:R\otimes_F R\to R\otimes_C
C[\uGal(R/F)]\end{equation}
is an $R$-linear ID-isomorphism. Again, the ring of regular functions $C[\uGal(R/F)]$ is equipped with
the trivial iterative derivation.

On the other hand, if $(R,\theta_R)$ is an ID-simple ID-ring extending $(F,\theta)$ with the same
constants, and if there is an $R$-linear ID-isomorphism $\gamma:R\otimes_F R\to R\otimes_C
C[{\mathcal G}]$ for some affine group scheme ${\mathcal G}\leq \GL_{n,C}$ corresponding to an
action of ${\mathcal G}$, then $R/F$ is indeed a Picard-Vessiot ring for some IDE (cf.
\cite{am:gticngg}, Prop.~10.12).

For later purposes, also keep in mind that for a finite Picard-Vessiot extension $R/F$, the PV-ring
$R$ already is a field. Hence, in that case the quotient field $E$ coincides with the PV-ring $R$.

\comment{
This torsor isomorphism (\ref{eq:torsor-iso}) is the key tool to establish the Galois
correspondence between
the closed subgroup schemes of $\G=\uGal(R/F)$ and the intermediate ID-fields
of the extension $E/F$, in more detail:
\begin{thm}{\bf (Galois correspondence)}\label{galois_correspondence}
Let $E/F$ be a PV-extension with PV-ring $R$ and Galois group scheme
$\G$.

Then there is an inclusion reversing bijection between
$$\fH:=\{ \H \mid \H\leq\G \text{ closed subgroup scheme of }\G
\}$$
and
$$\fM:=\{ M \mid F\leq M\leq E \text{ intermediate ID-field} \}$$
given by 
$\Psi:\fH \to \fM,\H\mapsto E^{\H}$ and 
$\Phi:\fM \to \fH, M\mapsto \underline{\Gal}(E/M)$.\\
With respect to this bijection, $\H\in \fH$ is a normal subgroup of $\G$, if and
only if $E^{\H}$ is a PV-field over $F$. In this case the Galois
group scheme $\underline{\Gal}(E^{\H}/F)$ is isomorphic to $\G/\H$.
\end{thm}

(See \cite{am:gticngg}, Thm.~11.5, resp. \cite{am:igsidgg}, Prop.~3.4 and 
Thm.~3.5 for the proof of this theorem.) 
The invariants $E^\H$ are defined to be all elements $e=\frac{r}{s}\in E$, such
that, for all 
$C$-algebras $D$, and all $h\in\H(D)$,
$$\frac{h(r\otimes 1)}{h(s\otimes 1)}=e\otimes 1\in\Quot(E\otimes_C D),$$
where $\Quot(E\otimes_C D)$ denotes the localisation of $E\otimes_C D$ by all
nonzerodivisors.
}

\section{Iterative derivations compatible with addition}\label{sec:it-der-on-ab-schemes}
 Let $C$ be a field of positive characteristic $p$, $k=C(t)$ the rational function field with
iterative derivation by $t$, and let $A/C$ be a connected abelian scheme over $C$. The addition map
on
$A$ will be denoted by $\oplus:A\times A\to A$ (and the subtraction by $\ominus$).\\
Let $K_A$ denote the function field of $A$. Let $(L,\theta)$ be the field $L=K_A(t)$
with some higher derivation $\theta$ extending the one on $k=C(t)$, and let $D$ be the field $K_A$
equipped with the trivial higher derivation. The higher derivations of $L$ and $D$ are extended to a
higher derivation (also
denoted by $\theta$) on $LD:=L\cdot D:=\Quot(L\otimes_C D)$.

The map $\oplus$ induces a homomorphism of fields $K_A\to K_{A\times A}=K_A\cdot K_A$ and also a homomorphism $L\to L\cdot D$ by $t$-linear extension. Extending again
$D$-linearly, we obtain an isomorphism $\rho:LD\to LD$. This isomorphism fixes exactly the
elements in $D(t)\subseteq LD$, i.e.~$D(t)=\{ x\in LD \mid \rho(x)=x\}$.
Actually $\rho$ is nothing else than the homomorphism on the generic fibers corresponding to
$A_{C(t)}\times A \to 
A_{C(t)}\times A, (p_1,p_2)\mapsto (p_1\oplus p_2,p_2)$.

\begin{lem}\label{lem:rho-HD-homo}
With notation as above, let $\eta_L\in A(L)$ be the generic point, and let $\theta_{*}:A(L)\to
A(L[[T]])$ be the map induced by $\theta$. Then
 $\rho$ is an HD-homomorphism if and only if
$\eta_L \ominus \theta_{*}(\eta_L)\in A(C(t)[[T]])$.
\end{lem}

\begin{proof} 

Since in any case $\eta_L \ominus \theta_{*}(\eta_L)\in A(L[[T]])$, the condition is equivalent to
saying that $\eta_L \ominus \theta_{*}(\eta_L)\in A(D(t)[[T]])\subseteq A(LD[[T]])$.

Let $\eta_D$ denote the generic point of $A$ in $A(D)$, and $\rho_*:A(LD)\to A(LD)$ the map induced
by $\rho$. Then by construction, one has $\rho_*(\eta_L)=\eta_L\oplus \eta_D$, and therefore,
$\theta_*(\rho_*(\eta_L))=\theta_*(\eta_L\oplus \eta_D)=\theta_*(\eta_L)\oplus \eta_D$, since
$\theta$ acts trivially on $D$.

Hence:
\begin{eqnarray*}
\eta_L \ominus \theta_{*}(\eta_L)\in A(D(t)[[T]]) &\LRa& (\rho[[T]])_*(\eta_L \ominus
\theta_{*}(\eta_L))=\eta_L \ominus \theta_{*}(\eta_L)\\
&\LRa& \rho_*(\eta_L) \ominus (\rho[[T]])_*(\theta_{*}(\eta_L))=\eta_L \ominus \theta_{*}(\eta_L)\\
&\LRa& (\eta_L\oplus \eta_D) \ominus (\rho[[T]])_*(\theta_{*}(\eta_L))=\eta_L \ominus
\theta_{*}(\eta_L)\\
&\LRa& \theta_*(\eta_L)\oplus \eta_D = (\rho[[T]])_*(\theta_{*}(\eta_L))\\
&\LRa& \theta_*(\rho_*(\eta_L))=(\rho[[T]])_*(\theta_{*}(\eta_L))
\end{eqnarray*} 
Since $\eta_L$ is the generic point of $A$, the last equality is equivalent to
$\theta\circ \rho =\rho[[T]]\circ \theta$, i.e.~to the condition that $\rho$ is an HD-homomorphism.
\end{proof}

\begin{thm}\label{thm:commuting with rho}
We use notation as above. Let $C(t)[[T,U]]$ be the power series ring over $C(t)$ in two variables $
T$ and $U$ and let $R$ denote the subring of $C(t)[[T,U]]$ of those power series $P(t,T,U)$ such
that $P(t+U,T,0)=P(t,T,U)$. 

Then $\theta$ is an iterative derivation and $\rho$ is an
ID-homomorphism if and only if $\theta_{U,*}(\eta_L)\ominus \theta_{T+U,*}(\eta_L)\in A(R)$.

As already mentioned earlier, $\theta_{U}:LD\to LD[[U]]$ and $\theta_{T+U}:LD\to LD[[T,U]]$
denote the maps $\theta$ with
$T$ replaced by $U$ and $T+U$, respectively, and $\theta_{U,*}:A(LD)\to A(LD[[U]])$ as well as $\theta_{T+U,*}:A(LD)\to A(LD[[T,U]])$ the induced maps.
\end{thm}

\begin{proof}
Let us first remark that $R$ is nothing else than the image of $C(t)[[T]]$ under the homomorphism $\theta_U[[T]]$, since the map $\theta_U[[T]]$ on $C(t)[[T]]$ is just replacing $t$ by $t+U$.\\
Now assume that $\theta$ is an iterative derivation such that $\rho$ is an ID-homo\-morphism. Since
$\theta$ is an iterative derivation, one has $\theta_{U+T}=\theta_U[[T]]\circ \theta_T$, and therefore
$\theta_{U,*}(\eta_L)\ominus \theta_{T+U,*}(\eta_L)=(\theta_{U}[[T]])_{*}\left( \eta_L\ominus
\theta_{T,*}(\eta_L) \right)$. Since $\rho$ is an ID-homomorphism, one has $\eta_L\ominus
\theta_{T,*}(\eta_L)\in A(C(t)[[T]])$ by the previous lemma. Hence, we obtain
$(\theta_{U}[[T]])_{*}\left( \eta_L\ominus \theta_{T,*}(\eta_L) \right)\in A(C(t)[[T,U]])$.
By the characterisation of $R$ above, the point $(\theta_{U}[[T]])_{*}\left( \eta_L\ominus \theta_{T,*}(\eta_L) \right)$ is indeed $R$-valued.

For the converse, let $\theta_{U,*}(\eta_L)\ominus \theta_{T+U,*}(\eta_L)\in A(R)$. Mapping $U$ to $0$ leads to 
$\eta_L\ominus \theta_{T,*}(\eta_L)\in  A(C(t)[[T]])$, hence $\rho$ is an HD-homomorphism by the
previous lemma. As before, the condition that the expression is in $A(R)$ implies that we obtain
the same element when mapping $U\mapsto 0$ and applying $(\theta_{U}[[T]])_{*}$. Hence
\begin{eqnarray*}
\theta_{U,*}(\eta_L)\ominus
\theta_{T+U,*}(\eta_L)&=&(\theta_{U}[[T]])_{*}\left( \theta_{0,*}(\eta_L)\ominus
\theta_{T+0,*}(\eta_L)\right)\\
&=&\theta_{U,*}(\eta_L)\ominus (\theta_{U}[[T]])_{*}\left(\theta_{T,*}(\eta_L)\right)
\end{eqnarray*}
This means $\theta_{T+U,*}(\eta_L)= (\theta_{U}[[T]])_{*}\left(\theta_{T,*}(\eta_L)\right)$. Since
$\eta_L$ is the generic point, this implies
$\theta_{T+U}=\theta_{U}[[T]]\circ \theta_{T}$, and therefore $\theta$ is iterative.
\end{proof}

\begin{rem}
So far, we didn't use commutativity of $\oplus$. Hence, all the statements made so far are also
valid for non-commutative connected group schemes instead of abelian schemes.
\end{rem}

\section{Torsion schemes as Galois group schemes}\label{sec:galois-groups}

We use the notation of the previous section. In particular, $A/C$ is an abelian scheme and $L$ is
the function field of $A_{C(t)}$ equipped with a higher derivation $\theta$ extending the
iterative derivation with respect to $t$ on $C(t)$.

\begin{thm}\label{thm:torsion as galois}
Let $\theta$ be an iterative derivation on $L$ such that $\rho$ is an ID-homomorphism. Also assume
that the constants of $(L,\theta)$ are $C$. For $n\in
\NN$, let $[n]:A\to A$ denote multiplication by $n$, $A[n]=\Ker([n])$ the $n$-torsion scheme, and
$[n]^\#:L\to L$ the corresponding map on the function fields of $A_{C(t)}$. Then
\begin{enumerate}
\item the subfield $[n]^\#(L)\subseteq L$ is an ID-subfield of $L$,
\item the extension $L/[n]^\#(L)$ is a PV-extension and the iterative differential Galois group
scheme is given as 
$$\uGal(L/[n]^\#(L)) \isom A[n]$$
as affine group schemes over $C$.
\end{enumerate}
\end{thm}

\begin{proof}
The addition $A\times A[n]\to A$ induces a homomorphism
$\bar{\rho}:\O_A(U)\to \O_A(U)\otimes_C C\bigl[A[n]\bigr]$ for an appropriate (affine) open
subset $U\subseteq A$, where $C\bigl[A[n]\bigr]$ denotes the ring of regular functions on the affine scheme $A[n]$. The subring $[n]^\#(\O_A(U))$ is then the equalizer of $\bar{\rho}$ and
$\id\otimes 1$.\\
Furthermore, $\bar{\rho}$ can be extended to a homomorphism $\bar{\rho}:L\to L\otimes_C C\bigl[A[n]\bigr]$ by
$\bar{\rho}(t)=t\otimes 1$ and by localisation. This map $\bar{\rho}$ is actually a specialisation
of the map $\rho:L\to LD$. By assumption $\rho$ is an ID-homomorphism and therefore $\bar{\rho}$ is
an ID-homomorphism when $C\bigl[A[n]\bigr]$ is equipped with the trivial iterative derivation.\\
This shows that the equalizer $[n]^\#(L)\subseteq L$ is an ID-subfield of $L$.

The $L$-linear extension of $\bar{\rho}$ leads to an ID-homomorphism
$\bar{\rho}_L:L\otimes_{[n]^\#(L)} L\to L\otimes_C C\bigl[A[n]\bigr]$ which is a monomorphism, since
$[n]^\#(L)$ is the equalizer of $\bar{\rho}$ and $\id\otimes 1$.\\
As the degree of the extension $L/[n]^\#(L)$ equals the dimension $\dim_C(C\bigl[A[n]\bigr])$, this 
monomorphism is indeed an ID-isomorphism.\\
Therefore, the second claim follows by \cite{am:gticngg}, Prop. 10.12. (Here we use that the
constants of $L$ are indeed $C$.)
\end{proof}

\section{Extension of iterative derivations}\label{sec:ID-extensions}

In this section we develop criteria for a higher derivation to be iterative. This will be used in the last section. We will assume that $C$ is a field of characteristic $p>0$, and $(F,\theta)$ is an
ID-field containing $C(t)$ such that $\theta|_{C(t)}$ is the iterative derivation with
respect to~$t$
(compare Ex.~\ref{ex:ID-fields}(\ref{item:der by t})).

\begin{thm}\label{thm:equivalent conditions}
Let $L$ be a finitely generated separable field extension of $F$ with a higher derivation on $L$
extending $\theta$ on $F$, which will also be denoted by $\theta$. Let $x_1,\dots, x_k$ be a
separating transcendence
basis of $L$ over $F$,
 and $\theta(x_i)=: x_i+ \sum_{n=1}^\infty \xi_{i,n} T^n$ for all
$i=1,\dots, k$.\\
Assume that $\xi_{i,n}\in L^pF\subset L$ for all $i=1,\dots, k$ and all $n\geq 1$. Then for any
$\ell_0\geq 0$ the following are equivalent:
\begin{enumerate}
\item The iteration rule holds on $L$ up to level $p^{\ell_0+1}$.
\item For all $0\leq \ell \leq \ell_0$, one has:
\begin{enumerate}
\item for all $0\leq m< p^\ell$ and $0<a<p$:
$\theta^{(m+ap^\ell)}=\frac{1}{a!}\left(\theta^{(p^\ell)}\right)^a\circ \theta^{(m)}$,
\item $\left(\theta^{(p^\ell)}\right)^p=0$, and
\item for all $0\leq j< \ell$: $\theta^{(p^j)}\circ
\theta^{(p^\ell)}=\theta^{(p^\ell)}\circ \theta^{(p^j)}$.
\end{enumerate}
\item The iteration rule holds up to level $p^{\ell_0+1}$ for all $x_i$ ($i=1,\dots, k$).
\item[(3')]  Condition (2) holds when evaluated at all $x_i$ ($i=1,\dots, k$).
\item\label{item:explicite formula} For all $0\leq \ell\leq \ell_0$ and $i=1,\dots k$, one has:
$$\xi_{i,p^\ell}+ \sum_{m=1}^{p^\ell-1} \theta^{(p^\ell)}(\xi_{i,m})(-t)^m \in
\bigcap_{0\leq j<\ell} \Ker\left(\theta^{(p^j)}\right) \cap \Ker\left(\theta^{(p^\ell(p-1))}\right)
,$$
for all $1<a<p$:
$\xi_{i,ap^\ell}=\frac{1}{a!}\left(\theta^{(p^\ell)}\right)^{a-1}(\xi_{i,p^\ell})$,
and for all $0<m< p^\ell$ and $0<a<p$:
$$\xi_{i,m+ap^\ell}=\frac{1}{a!}\left(\theta^{(p^\ell)}\right)^a(\xi_{i,m}).$$
\end{enumerate}
\end{thm}

\begin{rem}
Condition (\ref{item:explicite formula}) of the previous theorem, gives a recursive rule for
constructing an iterative derivation on $L$. In more detail:
\begin{enumerate}
\item Choose $\xi_{i,1}\in L^pF \cap \Ker\left(\theta^{(p-1)}\right)=L^p \bigl(F \cap
\Ker\left(\theta^{(p-1)}\right)\bigr)$ arbitrarily for all $i=1,\dots, k$.
\item Calculate $\xi_{i,a}:=\frac{1}{a!}\left(\theta^{(1)}\right)^{a-1}(\xi_{i,1})$ for $1<a<p$.
\item Proceed inductively: Assume that for $\ell>0$, the elements $\xi_{i,m}$ for $m<p^{\ell}$ are
already given satisfying condition (\ref{item:explicite formula}) of the theorem. Then choose
$$\xi_{i,p^\ell}\in - \sum_{m=1}^{p^\ell-1} \theta^{(p^\ell)}(\xi_{i,m})(-t)^m +
\bigcap_{0\leq j<\ell} \Ker\left(\theta^{(p^j)}\right) \cap \Ker\left(\theta^{(p^\ell(p-1))}\right)
\cap L^pF$$
and calculate $\xi_{i,ap^\ell}$ for $1<a<p$ as well as
$\xi_{i,m+ap^\ell}$ for $0<m< p^\ell$ and $0<a<p$, by
the rules above.\\ 
Since for an element $x^p\in L^p$, one has
$\theta^{(p^\ell)}(x^p)=\left(\theta^{(p^{\ell-1})}(x)\right)^p$, the condition $\xi_{i,m}\in
L^p F$ implies that $\theta^{(p^\ell)}(\xi_{i,m})$ is
computable using only the values $\xi_{i,m}$ for $m<p^\ell$. By the same reason the set
$\bigcap_{0\leq j<\ell} \Ker\left(\theta^{(p^j)}\right) \cap \Ker\left(\theta^{(p^\ell(p-1))}\right)
\cap L^p F$ is determined by the elements $\xi_{i,m}$ for $m<p^\ell$.
\end{enumerate}
\end{rem}

\begin{proof}[Proof of Thm.~\ref{thm:equivalent conditions}]
\begin{enumerate}
\item[(1)$\Leftrightarrow$(2)] All three conditions in (2) follow directly from the iteration rule
for $\theta$. On the other hand, given the conditions in (2),  any $\theta^{(i)}$ with $i< p^{\ell_0+1}$ can be written as a composition of several $\theta^{(p^n)}$ as in Lemma \ref{lem:known-stuff}(\ref{item:expansion}). Then it is not hard to check that
 $\theta^{(i)}\circ \theta^{(j)}$ indeed equals $\theta^{(i+j)}$ whenever $i+j<p^{\ell_0+1}$. Using again this decomposition and the conditions $\left(\theta^{(p^\ell)}\right)^p=0$, one verifies that $\theta^{(i)}\circ \theta^{(j)}=0$ whenever $i,j>0$ and $i+j=p^{\ell_0+1}$.
Hence the iteration rule holds on $L$ up to level $p^{\ell_0+1}$.
 \item[(1)$\Leftrightarrow$(3)] We only have to show that (3) implies (1). 
Since the set for which the iteration rule holds up to level $p^{\ell_0+1}$ is a subfield of
$L$ (cf.~Lemma \ref{lem:Things to show}(\ref{item:subfield})) and since
$x_1,\dots, x_k$ generate $F(x_1,\dots, x_k)$ over $F$ it is immediate that the iteration rule
holds up to level $p^{\ell_0+1}$ on $F(x_1,\dots, x_k)$. But an extension of an iterative derivation to a finite separable field
extension is unique, and again an iterative derivation. So the iteration rule holds on $L$ up to level $p^{\ell_0+1}$.
\item[(2)$\Leftrightarrow$(3')] This is shown in a similar way.
\item[(1),(2) $\Rightarrow$(4)] By the iteration rule resp. condition (2)(a), one has
$$\xi_{i,m+ap^\ell}=\theta^{(m+ap^\ell)}(x_i)=\frac{1}{a!}\left(\theta^{(p^\ell)}\right)^a\circ
\theta^{(m)}(x_i)= \frac{1}{a!}\left(\theta^{(p^\ell)}\right)^a(\xi_{i,m})$$
for all $0<m\leq p^\ell$ and $0<a<p$ s.t. $m+ap^\ell<p^{\ell +1}$. Furthermore for all $1\leq m\leq  p^\ell-1$, 
\begin{eqnarray*}
\theta^{(p^\ell)}\left( \theta^{(m)}(x_i)(-t)^m\right)
&=& \sum_{k=0}^{p^\ell} \theta^{(k)}\theta^{(m)}(x_i)  (-1)^m\theta^{(p^\ell-k)}(t^m) \\
&=&  \sum_{k=0}^{p^\ell} \binom{k+m}{k}\theta^{(k+m)}(x_i) (-1)^m \binom{m}{p^\ell-k} t^{m-p^\ell+k}\\
&=& \theta^{(p^\ell+m)}(x_i) (-1)^m t^m=\theta^{(p^\ell)}(\xi_{i,m})(-t)^m,
\end{eqnarray*}
as for $k<p^\ell-m$ the second binomial coefficient vanishes and for $p^\ell>k\geq p^\ell-m$ the first one. Hence, using condition (2)(c)
and Lemma \ref{lem:Things to show}(\ref{item:cancellation}) we have
$$\theta^{(p^j)}\left( \xi_{i,p^\ell}+ \sum_{m=1}^{p^\ell-1}
\theta^{(p^\ell)}(\xi_{i,m})(-t)^m\right)=
\theta^{(p^\ell)}\theta^{(p^j)} \left(x_i+ \sum_{m=1}^{p^\ell-1} \theta^{(m)}(x_i)(-t)^m\right)=0$$
for all $0\leq j<\ell$ and by condition (2)(b)
$$\theta^{(p^{\ell}(p-1))}\left( \xi_{i,p^\ell}+ \sum_{m=1}^{p^\ell-1}
\theta^{(p^\ell)}(\xi_{i,m})(-t)^m\right)=\theta^{(p^{\ell}(p-1))}\theta^{(p^\ell)}
\left(x_i+ \sum_{m=1}^{p^\ell-1} \theta^{(m)}(x_i)(-t)^m\right)=0.$$
\item [(4)$\Rightarrow$(3')] 
The formulae for $\xi_{i,ap^\ell}$ and $\xi_{i,m+ap^\ell}$ imply the conditions (2)(a)
evaluated at $x_i$. Furthermore, by induction
$\theta^{(p^{j-1})}\theta^{(p^{\ell-1})}=\theta^{(p^{\ell-1})}\theta^{(p^{j-1})}$ for all $j<\ell$
and hence $\theta^{(p^j)}\theta^{(p^{\ell})}(x)=\theta^{(p^{\ell})}\theta^{(p^j)}(x)$ for all $x\in
L^pF$, in particular for $x=\theta^{(m)}(x_i)$. This implies
\begin{eqnarray*}
0&=&\theta^{(p^j)}\left( \xi_{i,p^\ell}+ \sum_{m=1}^{p^\ell-1}
\theta^{(p^\ell)}(\xi_{i,m})(-t)^m\right) \\
&=&\theta^{(p^j)}( \xi_{i,p^\ell})-\theta^{(p^\ell)}\theta^{(p^j)}(x_i)+
\theta^{(p^\ell)}\theta^{(p^j)}(x_i)+ \theta^{(p^j)} \left(\sum_{m=1}^{p^\ell-1}
\theta^{(p^\ell)}(\theta^{(m)}(x_i))(-t)^m\right)\\
&=& \theta^{(p^j)}\theta^{(p^\ell)}(x_i)-\theta^{(p^\ell)}\theta^{(p^j)}(x_i).
\end{eqnarray*}
The last step is obtained by the same calculation as above.

Similarly, one obtains
\begin{eqnarray*}
0&=&\theta^{(p^\ell(p-1))}\left( \xi_{i,p^\ell}+ \sum_{m=1}^{p^\ell-1}
\theta^{(p^\ell)}(\xi_{i,m})(-t)^m\right) \\
&=&\theta^{(p^\ell(p-1))} \theta^{(p^\ell)}(x_i)+
\sum_{m=1}^{p^\ell-1}\theta^{(p^\ell)}\theta^{(p^\ell(p-1))}(\xi_{i,m})(-t)^m\\
&=& \frac{1}{(p-1)!}(\theta^{(p^\ell)})^{p-1}\theta^{(p^\ell)}(x_i).
\end{eqnarray*}
\end{enumerate}
\end{proof}

\section{Example}\label{sec:example}

In this section we give an example to illustrate the previous sections.
In this example it is even
possible to give a recursive formula for constructing an iterative derivation $\theta$ which is compatible with the addition map (see Theorem
\ref{thm:strong formula}).
Indeed, it will be a sharpening of the formula in Thm.~\ref{thm:equivalent
conditions}, Item \ref{item:explicite formula}.

The example we consider is the elliptic curve $E/C$ in characteristic $p=2$ given by the equation
$x^3=z^2+z$, the neutral element of addition being given by the point $(0,0)$.

As before, $K_E/C$ denotes the function field of $E/C$, and $L=K_E(t)=C(x,z,t)$ is the HD-field with
a higher
derivation $\theta$ extending the iterative derivation with respect to $t$ on $C(t)$. The iterative
derivatives of $x$ are denoted by $\xi_m$, i.e.~$\theta(x)=:x+
\sum_{m=1}^\infty \xi_{m} T^m$, and $\eta_L:=(x,z)\in E(L)$ is the generic point of $E$.
 Furthermore,
$D=K_E$ denotes the ID-field with trivial iterative derivation.

\begin{lem}
For two points $(x_1,z_1)$ and $(x_2,z_2)$ the difference $(x_d,z_d):=(x_1,z_1)\ominus (x_2,z_2)$ is
given by:
$$x_d=x_2+\frac{x_1}{1+z_1}+\left( \frac{z_2- \frac{z_1}{1+z_1}}{x_2- \frac{x_1}{1+z_1}}\right)^2$$
and
$$z_d=\frac{z_2-\frac{z_1}{1+z_1}}{x_2-\frac{x_1}{1+z_1}}\cdot (x_d-x_2)+z_2$$
\end{lem}

\begin{proof}
One only has to check, that the point $(x_d,z_d)$ is the third intersection of the elliptic curve
with the line passing through $(x_2,z_2)$ and $\ominus
(x_1,z_1)=(\frac{x_1}{1+z_1},\frac{z_1}{1+z_1})$.
\end{proof}

Let 
$$f(T):=\sum_{k=0}^\infty f_mT^m :=\theta(x)+\frac{x}{1+z}+ \left(  \frac{\theta(z)-
\frac{z}{1+z}}{\theta(x)- \frac{x}{1+z}}\right)^2\in L[[T]].$$
Then by the previous lemma, $f(T)$ is the $x$-coordinate of $\eta_L\ominus\theta_*(\eta_L)$.
For the coefficients $f_m$ we have: $f_0=0$, $f_m=\xi_m$ for odd $m$ and
$f_m=\xi_m+({\widetilde{f}_m})^2$ for even $m>0$ and an
appropriate element $\widetilde{f}_m\in L$, depending only on $x,z$ and the elements $\xi_k$ for
$k\leq m/2$. 

Furthermore, let $g(T)$ denote the $z$-coordinate of $\eta_L\ominus\theta_*(\eta_L)$,
i.e.~$\eta_L\ominus\theta_*(\eta_L)=(f(T),g(T))$ in these local coordinates.
Since this is a point on $E$, one has the relation $f(T)^3=g(T)^2+g(T)$, and hence the coefficients
$g_m$ of $g(T)=:\sum_{k=0}^\infty g_mT^m$ can be expressed in terms of the $f_m$. In more detail,
$g_0=g_1=g_2=0$ and $g_m$ can be written as a polynomial in $f_1,\dots, f_{m-2}$.

\begin{lem}\label{lem:theta-f}
Assume that $\theta$ is an iterative derivation on $L$. Then for even $m,j\in\NN\setminus \{0\}$ the difference
$\theta^{(m)}(f_j)-\binom{m+j}{m}f_{m+j}$ is a polynomial in
$\binom{(m+j)/2}{m/2}f_{(m+j)/2}$, $f_{(m+j)/2-1},\dots, f_1$, whereas for all other choices of
$m,j\in \NN$ this difference is $0$.
\end{lem}

\begin{proof}
By definition, $f(T)$ is the $x$-coordinate of $\eta_L\ominus\theta_*(\eta_L)$, hence
$\theta_U[[T]](f(T))$ is the $x$-coordinate of $(\theta_{U}[[T]])_* (\eta_L\ominus\theta_*(\eta_L))$.
But
\begin{eqnarray*}
(\theta_{U}[[T]])_*(\eta_L\ominus\theta_*(\eta_L)) &=& \theta_{U,*}(\eta_L)\ominus
(\theta_{U}[[T]])_*(\theta_*(\eta_L)) \\
&=& \theta_{U,*}(\eta_L)\ominus \eta_L \oplus \eta_L\ominus\theta_{U+T,*}(\eta_L)\\
&=& \bigl(\eta_L\ominus\theta_{U+T,*}(\eta_L) \bigr)\ominus \bigl(\eta_L
\ominus \theta_{U,*}(\eta_L)\bigr)
\end{eqnarray*}
Hence, $(\theta_U[[T]](f(T)),\theta_U[[T]](g(T)))=(f(U+T),g(U+T))\ominus (f(U),g(U))$.
Using the formula for the difference, we obtain
$$\theta_U[[T]](f(T))=f(U)+\frac{f(T+U)}{1+g(T+U)} +\left(
\frac{g(U)-\frac{g(T+U)}{1+g(T+U)}}{f(U)-\frac{f(T+U)}{1+g(T+U)}} \right)^2\in L[[T,U]].$$
The coefficient of $U^mT^j$ on the left hand side is $\theta^{(m)}(f_j)$. 
For the right hand side, we first remark that
$$\frac{f(T+U)}{1+g(T+U)}=f(T+U)+\bigl( g(T+U)/(T+U)\bigr)^2 \cdot \bigl( f(T+U)/(T+U)\bigr)^{-2},$$
as power series in $(T+U)$. So the right hand side is $f(U)+f(T+U)$ modulo squares.
This already shows that the coefficient of $U^mT^j$ on the right hand
side is $\binom{m+j}{m}f_{m+j}$, if $m$ or $j$ are odd.

For the other coefficients one has to have a closer look at the equation. Therefore, we consider
the remaining terms as power series in $(T+U)$ with coefficients in $L((U))$. The coefficient
of $U^mT^j$ in $\bigl( g(T+U)/(T+U)\bigr)^2 \cdot \bigl( f(T+U)/(T+U)\bigr)^{-2}$ is
$\binom{m+j}{m}$ times the coefficient of $(T+U)^{m+j}$ in this expression. Since $\bigl(
g(T+U)/(T+U)\bigr)$ is a multiple of $(T+U)^2$, this coefficient 
depends only on $f_{(m+j)/2-2}, f_{(m+j)/2-3},\dots, f_1$.
The last term in the equality above is the square of
\begin{eqnarray*}
\frac{g(U)-\frac{g(T+U)}{1+g(T+U)}}{f(U)-\frac{f(T+U)}{1+g(T+U)}} 
&=& \frac{1}{1+g(U)}\cdot \frac{g(U)+g(U)^2+(g(U)^2-1)g(T+U)}{f(U)+f(U)g(T+U)-f(T+U)}\\
&=&\frac{1}{1+g(U)}\cdot \frac{f(U)^3+(g(U)^2-1)g(T+U)}{f(U)+f(U)g(T+U)-f(T+U)}\\
&=& \frac{1}{1+g(U)} f(U)^2\cdot \frac{ 1+ \sum_{k=1}^\infty g_k(\frac{g(U)^2-1}{f(U)^3})
(T+U)^k }{ 1+ \sum_{k=1}^\infty (g_k -\frac{f_k}{f(U)})(T+U)^k }\\
&=& \bigl(1+g(U)\bigr)^{-1} f(U)^2\cdot \left(\sum_{n=0}^\infty \tau_n (T+U)^n \right)
\end{eqnarray*}
where $\tau_n$ is some polynomial in $f_1,\dots, f_n$ (and $g_1,\dots, g_n$), $g(U)$ and
$\frac{1}{f(U)}$.
Since the whole expression is a power series, $f(U)^2\cdot \tau_n$ is already in $L[[U]]$. Hence,
the coefficient of $U^mT^j$ in $\left( \frac{1}{1+g(U)} f(U)^2\cdot (\sum_{n=0}^\infty \tau_n
(T+U)^n )\right)^2$ depends only on $f_{(m+j)/2}, f_{(m+j)/2-1},\dots, f_1$, and $f_{(m+j)/2}$ only
occurs with the factor $\binom{(m+j)/2}{m/2}$.
\end{proof}

\begin{thm}\label{thm:strong formula}
$\theta$ is an iterative derivation on $L$ commuting with $\rho$ if and only if for all $\ell\geq
0$ and all $0<m<2^\ell$  one has $\xi_{m+2^\ell}=\theta^{(2^\ell)}(\xi_m)$ and
\begin{equation*}\label{eqn:formel}
\xi_{2^\ell}\in \sum_{m=1}^{2^\ell-1} \theta^{(2^\ell)}(\xi_m)t^m
+ \left( \sum_{m=0}^{2^\ell-1} \theta^{(m)}(\widetilde{f}_{2^\ell})t^m\right)^2 +
C(t^{2^{\ell+1}}). \tag{$*_\ell$} 
\end{equation*}
In particular, it is possible to choose/calculate elements $\xi_m$ recursively for
$m=1,2,\dots $ in order to obtain an iterative derivation on $L$ commuting with $\rho$.
\end{thm}

\begin{proof}
First let $\theta$ be an iterative derivation which commutes with $\rho$. Then
$\xi_{m+2^\ell}=\theta^{(m)}(\xi_{2^\ell})$ for all $0<m<2^\ell$ by Theorem \ref{thm:equivalent
conditions}.\\
Further using the rules in Theorem \ref{thm:equivalent
conditions}, we obtain:
\begin{eqnarray*}
\xi_{2^\ell}+ \sum_{m=1}^{2^\ell-1} \theta^{(2^\ell)}(\xi_m)t^m 
&=& \sum_{m=1}^{2^\ell-1} \theta^{(m)}(\xi_{2^\ell})t^m
= \sum_{m=1}^{2^{\ell+1}-1} \theta^{(m)}(\xi_{2^\ell})t^m,
\end{eqnarray*}
since $\theta^{(m)}(\xi_{2^\ell})=0$ for $2^\ell\leq m<2^{\ell+1}$, as well as
\begin{eqnarray*}
\left( \sum_{m=0}^{2^\ell-1} \theta^{(m)}(\widetilde{f}_{2^\ell})t^m\right)^2
&=& \sum_{m=0}^{2^\ell-1} \left(\theta^{(m)}(\widetilde{f}_{2^\ell})\right)^2 t^{2m}
= \sum_{m=0}^{2^\ell-1} \theta^{(2m)}\bigl((\widetilde{f}_{2^\ell})^2\bigr) t^{2m}\\
&=& \sum_{m=0}^{2^{\ell+1}-1} \theta^{(m)}\bigl((\widetilde{f}_{2^\ell})^2\bigr) t^{m},
\end{eqnarray*}
since $\theta^{(m)}\bigl((\widetilde{f}_{2^\ell})^2\bigr)=0$ for $m$ odd. Combining these we get:
\begin{eqnarray*}
&&\xi_{2^\ell}+ \sum_{m=1}^{2^\ell-1} \theta^{(2^\ell)}(\xi_m)t^m +\left( \sum_{m=0}^{2^\ell-1}
\theta^{(m)}(\widetilde{f}_{2^\ell})t^m\right)^2\\
&&\qquad = \sum_{m=0}^{2^{\ell+1}-1} \theta^{(m)}(\xi_{2^\ell})t^m+\sum_{m=0}^{2^{\ell+1}-1}
\theta^{(m)}\bigl((\widetilde{f}_{2^\ell})^2\bigr) t^{m} 
 = \sum_{m=0}^{2^{\ell+1}-1} \theta^{(m)}(f_{2^\ell})t^m
\end{eqnarray*}
This expression is in $C(t)$, since $f_{2^\ell}\in C(t)$ by Lemma \ref{lem:rho-HD-homo}, and it is
in the intersection
$\bigcap_{0\leq j<\ell+1} \Ker(\theta^{(p^j)})$ by Lemma \ref{lem:Things to show}, hence in
$C(t^{2^{\ell+1}})$ as desired.

On the other hand, assume that the conditions on $\xi_{m+2^\ell}$ and on $\xi_{2^\ell}$ hold.
We will first show that $\theta$ is an iterative derivation by showing inductively that $\theta^{(j)}\circ
\theta^{(m)}=\binom{j+m}{j}\theta^{(j+m)}$ for all $j+m\leq 2^{\ell_0+1}$.

For $\ell_0=0$, condition ($*_{\ell_0}$) is just $\xi_1\in C(t^2)$, which implies
$\theta^{(1)}(\xi_1)=0$. Hence by Theorem \ref{thm:equivalent conditions}, the iteration rule holds
for all $j+m\leq 2=2^{0+1}$. 
Now, assume by induction that the iteration rule holds for all $j+m\leq 2^{\ell_0}$. Then it even
holds for all $j+m<2^{\ell_0+1}$, since $\xi_{m+2^\ell}=\theta^{(2^\ell)}(\xi_m)$, and we obtain by
Lemma \ref{lem:Things to show} that $\theta^{(2^j)}\left(\sum_{m=0}^{2^{\ell_0}-1}
\theta^{(m)}(x)t^m\right)=0$ for all $x\in L$ and $0\leq j<\ell_0$, in particular 
$\theta^{(2^j)}\left(\sum_{m=0}^{2^{\ell_0}-1}
\theta^{(m)}(\widetilde{f}_{2^{\ell_0}})t^m\right)=0$ for $0\leq j<\ell_0$.
Therefore using ($*_{\ell_0}$), $\xi_{2^{\ell_0}}+ \sum_{m=1}^{2^{\ell_0}-1}
\theta^{(2^{\ell_0})}(\xi_m)t^m\in \bigcap_{0\leq j\leq \ell_0} \Ker(\theta^{(2^j)})$.
By Theorem \ref{thm:equivalent conditions}, this shows that the iteration rule holds for $j+m\leq
2^{\ell_0+1}$.

It remains to show that $\rho$ is an ID-homomorphism. By Lemma \ref{lem:rho-HD-homo}, this is
equivalent to $f_k\in C(t)$ for all $k\geq 1$. Again we use induction: The case $k=1$ is given
by condition ($*_0$), since $f_1=\xi_1$.
Assume $f_m\in C(t)$ is already shown for $1\leq m \leq 2^\ell-1$.
If $k=2^\ell+m$ for some $0<m<2^\ell$, then by Lemma
\ref{lem:theta-f}, $f_{2^\ell+m}$ differs from $\theta^{(2^\ell)}(f_m)$ by a polynomial in
$f_j$ for $1\leq j\leq 2^\ell-1$, and hence is an element of $C(t)$ by induction.
If $k=2^\ell$, condition $(*_\ell)$ and the calculations above imply that
\begin{equation*}\sum_{m=0}^{2^{\ell+1}-1} \theta^{(m)}(f_{2^\ell})t^m\in C(t). \tag{$\dagger$}
\end{equation*}

By using Lemma \ref{lem:theta-f}, and $f_{2^\ell+m}\in C(t)$ as well as $f_j\in C(t)$ for $1\leq j\leq 2^\ell-1$, we see that $\theta^{(m)}(f_{2^\ell})$ is an
element of $C(t)$ for $0<m<2^\ell$, and also $\theta^{(2^\ell)}(f_{2^\ell})\in C(t)$, since $\binom{2^{\ell+1}}{2^\ell}$ and
$\binom{2^{\ell}}{2^{\ell-1}}$ are both zero in characteristic $2$. For $2^\ell<m<2^{\ell+1}$, we have
$\theta^{(m)}(f_{2^\ell})=\theta^{(m-2^\ell)}\left(\theta^{(2^\ell)}(f_{2^\ell})\right)\in C(t)$.
Therefore all the terms in ($\dagger$) different from $f_{2^\ell}$ are in $C(t)$ and hence
$f_{2^\ell}\in C(t)$.
\end{proof}

\comment{

} 

\bibliographystyle{plain}


\begin{thebibliography}{1}

\bibitem{td:tipdgtfrz}
Tobias Dyckerhoff.
\newblock The inverse problem of differential {G}alois theory over the field
  {R}(z).
\newblock Preprint, 2008.

\bibitem{am:gticngg}
Andreas Maurischat.
\newblock Galois theory for iterative connections and nonreduced {G}alois
  groups.
\newblock {\em Trans. Amer. Math. Soc.}, 362(10):5411--5453, 2010.

\bibitem{am:igsidgg}
Andreas Maurischat.
\newblock Infinitesimal group schemes as iterative differential {G}alois
  groups.
\newblock {\em J. Pure Appl. Algebra}, 214(11):2092--2100, 2010.

\bibitem{ar:icac}
Andreas R{\"o}scheisen.
\newblock {\em Iterative Connections and Abhyankar's Conjecture}.
\newblock PhD thesis, Heidelberg University, Heidelberg, Germany, 2007.

\end{thebibliography}

\vspace*{.5cm}

\parindent0cm

\end{document}